\documentclass[12pt]{amsart}
\usepackage{geometry}
\usepackage{hyperref}
\hypersetup{colorlinks=true,linkcolor=blue,urlcolor=blue,citecolor=red}
\newcommand{\RNum}[1]{\uppercase\expandafter{\romannumeral #1\relax}}
\usepackage{setspace}
\usepackage{amsmath}
\usepackage{amsthm}

\newtheorem{theorem}{Theorem}[section]
\newtheorem{remark}{Remark}[section]
\usepackage{pdfpages}
\usepackage{amssymb,latexsym}
\newcommand{\beq}{\begin{equation}}
\newcommand{\eeq}{\end{equation}}
\newtheorem{lemma}{Lemma}[section]
\newtheorem{property}{Property}

\newtheorem{corollary}{Corollary}[section]

\newcommand{\paren}[2]{\genfrac{(}{)}{0pt}{0}{#1}{#2}}

\begin{document}

\noindent A COMBINATORIAL ANALYSIS OF HIGHER ORDER GENERALISED GEOMETRIC POLYNOMIALS: A GENERALISATION OF BARRED PREFERENTIAL ARRANGEMENTS
\vspace{1cm} 

\noindent Sithembele Nkonkobe$^a$, Be\'{a}ta B\'{e}nyi$^b$, Roberto B. Corcino$^c$, Cristina B. Corcino$^d$
\tiny

\vspace{2mm}
\noindent $^a$Department of Mathematical Sciences, Sol Plaatje University, Kimberly, 8301, South Africa,  snkonkobe@gmail.com

\noindent$^b$Faculty of Water Sciences, National University of Public Service, Hungary, beata.benyi@gmail.com

\noindent$^c$Research Institute for Computational Mathematics and Physics, Cebu Normal University,  Cebu City, Phillipines,\\\indent 600, rcorcino@yahoo.com

\noindent$^d$Mathematics Department, Cebu Normal University, Cebu City, Phillipines, 6000, cristinacorcino@yahoo.com


\begin{minipage}[t]{16cm}

\fontsize{3}{1}

 \end{minipage}
 \normalsize
\section*{Abstract}
\noindent A barred preferential arrangement is a preferential arrangement, onto which in-between the blocks of the preferential arrangement a number of identical bars are inserted. We offer a generalisation of barred preferential arrangements by making use of the generalised Stirling numbers proposed by Hsu and Shiue (1998). We discuss how these generalised barred preferential arrangements offer a unified combinatorial interpretation of geometric polynomials. We also discuss asymptotic properties of these numbers.  

\noindent\fontsize{10}{1}\selectfont     
Mathematics Subject Classifications : 05A15, 05A16, 05A18, 05A19,   11B73, 11B83
\\\textbf{Keyword(s)}:preferential arrangement, barred preferential arrangement, geometric polynomial.
\normalsize
\vspace{2mm}


\begin{section}{Introduction}\end{section}

A \textit{barred preferential arrangement} of an $n$-element set $X_n$ is a preferential arrangement(ordered set partition) on which a number of bars are inserted in-between the formed blocks of $X_n$. 

The following are two examples of barred preferential arrangements of $X_5$ having two and three bars respectively, 

\RNum{1}).\:\:  $24\:\:  | 1 \quad53  | $

\RNum{2}). $|\:\: 2\:\: | 5\quad 4|\: 3\quad 1.$ 

The barred preferential arrangement in \RNum{1}, has two bars hence three sections. The first section from(left to right) has a single block i.e the block formed by the elements \{2,4\}. The second section (the one between the two bars) has two blocks which are, \{1\} and,  \{3,5\}. The third section(to the right of the second bar) is empty. The barred preferential arrangement in \RNum{2} has three bars hence four sections. The first section is empty. The second section has a single block. The third section has two blocks, and the fourth section also has two blocks. Barred preferential arrangements(BPA) having multiple bars seem to first appear in \cite{barred:2013}.

The generalised Stirling numbers $S(n,k,\alpha,\beta,\gamma)$ are defined in \cite{A unified approach to generalized Stirling numbers} in the following way, $(t|\alpha)_n=\sum\limits_{k=0}^{n}S(n,k,\alpha,\beta,\gamma)(t-\gamma|\beta)_n$, such that $(t|\alpha)_n$ is the generalised factorial polynomial $(t|\alpha)_n=\prod\limits_{k=0}^{n-1}(t-k\alpha)$, where $n\geq1$ and $\alpha,\beta,\gamma$ real or complex not all equal to zero.  
\vspace{3mm}

In \cite{Higher order geometric polynomials} the authors proposed as a way of generalising geometric polynomials proposed the following generating function; 
\begin{equation}\label{equation:1}\sum\limits_{n=0}^{\infty} T_n^{\lambda,x}(\alpha,\beta,\gamma)\frac{t^n}{n!}=
       \frac{(1+\alpha t)^{\gamma/\alpha}}{(1-x((1+\alpha t)^{\beta/\alpha}-1))^{\lambda}}
       .\end{equation}
       
Using  non-combinatorial methods the authors  recognised $T_n^{\lambda,x}(\alpha,\beta,\gamma)$ as generalisation of the counting sequence of barred preferential arrangements. The authors ask for a study of combinatorial properties of these numbers, which we provide in the current paper. 

 Nelsen and Schmidt proposed the family of generating functions (see~\cite{Chains in power set Nelsen 91}),\begin{equation}\label{equation:15} \frac{e^{\gamma x}}{2-e^x}.\end{equation} In the manuscript for $\gamma=2$  Nelsen and Schmidt interpreted the generating function as being that of the number of chains in the power set of $X_n$. The generating function for $\gamma=0$ is known to be that of number of preferential arrangements/number of outcomes in races with ties (see~\cite{gross:1962,mendelson:1982}). In the manuscript Nelsen and Schmidt then asked, ``could there be combinatorial structures associated with either $X_n$ or the power set of $X_n$ whose integer sequences are generated by  members of the family in \eqref{equation:15}, for other values of $\gamma$?" We will now refer to this question as the Nelsen-Schmidt question. In answering the  Nelsen-Schmidt question, the authors in \cite{Nkonkobe & Murali Nelesn-Schmidt generating function,Our paper on Generalised barred preferential arrangements}, offered combinatorial interpretations of integer sequences arising from the more general generating functions $\frac{e^{\gamma t}}{(2-e^t)^{\lambda}}$ and $\frac{e^{\gamma t}}{(2-e^{\beta t})^{\lambda}}$, where $\lambda,\beta,\gamma$ are in $\mathbb{N}_0$ (non-negative integers). In this study we offer a further generalisation of their results in answering the Nelsen-Schmidt question by interpreting combinatorially integer sequences arising from the generating function $\frac{(1+\alpha t)^{\gamma/\alpha}}{(1-x((1+\alpha t)^{\beta/\alpha}-1))^{\lambda}}$. Both generating functions $\frac{e^{\gamma t}}{(2-e^t)^{\lambda}}$, $\frac{e^{\gamma t}}{(2-e^{\beta t})^{\lambda}}$ arise as special cases of the generating function $\frac{(1+\alpha t)^{\gamma/\alpha}}{(1-x((1+\alpha t)^{\beta/\alpha}-1))^{\lambda}}$.

Geometric polynomials go far back as Euler's work on the year 1755 (see page 389 on Part \RNum{2} of \cite{Euler's work on geometric polynomials}). These polynomials are  well known in the literature, and it is also well known that these polynomials arise from variations of the  generating function $\frac{e^{r t}}{(1-x(e^{t}-1))^r}$, for instance in \cite{On generalised Bell polynomials,A series transformation formula,Apostol-Bernoulli functions,close encounters with stirling numbers of the second kind,Investing Geometric and exponential polynomials,Geometric Polynomials: Properties and Applications to Series with Zeta Values,Polynomials Related to Harmonic Numbers and Evaluation of Harmonic Number Series II,Fourier series of sums of products of ordered Bell and poly-Bernoulli functions,An explicit formula for Bernoulli polynomials in terms of $ r $-Stirling numbers of the second kind,Generalization of Mellin derivative and its applications}. In this study our combinatorial interpretation of the integer sequences arising from the generating function $\frac{(1+\alpha t)^{\gamma/\alpha}}{(1-x((1+\alpha t)^{\beta/\alpha}-1))^{\lambda}}$ offers a unified combinatorial interpretation of these geometric polynomials.

\begin{section}{When one bar is used.}

In this section we study the numbers $T_n^{1,x}(\alpha,\beta,\gamma)$. 




      \begin{lemma}\cite{Corcino Rorberto 2001}\label{lemma:3} For real/complex $\alpha,\beta,\gamma$, \fontsize{10}{1}
      
      $S(n,k,\alpha,\beta,\gamma)=\frac{1}{\beta^kk!}\Delta^k(\beta k+\gamma|\alpha)_n\big|_{s=0}=\frac{1}{\beta^kk!}\sum\limits_{s}^{}(-1)^{k-s}\binom{i}{s}(\beta s+\gamma|\alpha)_n$.
      \end{lemma}
      
      \begin{lemma}\label{lemma:10}\cite{Higher order geometric polynomials}For $\alpha,\beta,\gamma,x\in\mathbb{N}_0$ such that $(\alpha,\beta,\gamma,x)\not=(0,0,0,0)$,
      \begin{equation}\label{equation:9}
      G_n^x(\alpha,\beta,\gamma)=\sum\limits_{k=0}^{n}\beta^kk!x^kS(n,k,\alpha,\beta,\gamma).
      \end{equation}
      \end{lemma}
      \begin{lemma}\label{lemma:11}\cite{Higher order geometric polynomials,On generalised Bell polynomials}For real/complex $\alpha,\beta,\gamma,x$ such that $(\alpha,\beta,\gamma,x)\not=(0,0,0,0)$,
      \begin{equation}\label{equation:10}
      \sum\limits_{n=0}^{\infty}T_n^{1,x}(\alpha,\beta,\gamma)=\frac{(1+\alpha t)^{\gamma/\alpha}}{1-x((1+\alpha t)^{\beta/\alpha}-1)}.
      \end{equation}
      \end{lemma}
      \begin{theorem} \label{Theorem:15}
       \label{theorem:10}\cite{Paper with interpretation of G-stirling numbers}
     Given $\alpha,\beta,\gamma$ are non-negative integers such that $\alpha$ divides both $\beta$, and $\gamma$. Given $k+1$ distinct cells such that the first $k$ cells each contains $\beta$ labelled compartments, and the $(k+1){th}$ cell contains $\gamma$ labelled compartments. In each cell the compartments are having cyclic ordered numbering. The capacity of each compartment is limited to one ball. The number $\beta^k k!S(n,k,\alpha,\beta,\gamma)$, is the number of ways of distributing $n$ distinct elements into the $k+1$ cells one ball at a time such that only the cell having $\gamma$ compartments may be empty. compartments.
      \end{theorem}

We extend a special case of Theorem~\ref{Theorem:15} to the following property (see \cite{On generalised Bell polynomials}).

 \begin{property}\label{property:1} The number of ways of distributing $n$ balls into $k$ distinct cells one ball at a time (where $k$ runs from o to $n$), such that all the cells are non-empty, and each cell  has $\beta$ labelled compartments having cyclic ordered numbering such that for each consecutive available $\alpha$ compartments only the first compartment gets a ball, where each of the $k$ cells is colored with one of $x$ available colors is,
  $\sum\limits_{k=0}^{n}\beta^kk!x^kS(n,k,\alpha,\beta,0)=G_n^x(\alpha,\beta,0)$.  
               \end{property}

        \begin{property}\cite{Paper with interpretation of G-stirling numbers}\label{property:2}
                          Given a single cell with $\gamma$ compartments such that $\alpha|\gamma$, where the cells are given cyclic ordered numbering such that on each consecutive $\alpha$  compartments only the first compartment gets a ball. For any given $n$ balls the number of ways of placing the balls into the cell, one ball at a time is
                         $(\gamma|\alpha)_n$. \end{property}      
                     
                  \begin{remark}
                  It is clear that whether you consider barred preferential arrangements as a result of first forming blocks of elements and then inserting bars, or you first place the bars and then distribute the elements into the resultant sections, the results are the same. In our arguments in Theorems~\ref{theorem:2} to Theorem~\ref{theorem:12}, the latter way of viewing barred preferential arrangements is used, this technique is also used in \cite{Nkonkobe & Murali Nelesn-Schmidt generating function,Our paper on Generalised barred preferential arrangements}. In Theorem \ref{theorem:8} and \ref{theorem:11}, the former way of looking at barred preferential arrangements is used, this is the same technique used by the authors in \cite{barred:2013}.
                  \end{remark}   
                    \begin{remark}
                    In our arguments in the following theorems we assume that the section having property~\ref*{property:2} i.e the one having $\gamma$ compartments is the first section from(left to right). We will sometimes refer to this section as the special section. 
                    \end{remark}
               \begin{theorem}\label{theorem:2}
               The generating function $\frac{(1+\alpha t)^{\gamma/\alpha}}{1-x((1+\alpha t)^{\beta/\alpha}-1)}$ for $\alpha,\beta,\gamma$ in non-negative integers such that $\alpha|\beta$ and $\alpha|\gamma$, is that of the number of barred preferential arrangements with one bar, such that one section has property~\ref{property:2} and the other section has property~\ref{property:1}.
               \end{theorem}                   \begin{proof}
               By equation~\ref{equation:10} we have,
               
               \begin{equation}\label{equation:12}
               T_n^{1,x}(\alpha,\beta,\gamma)=\sum\limits_{k=0}^{n}\binom{n}{k}G_k^x(\alpha,\beta,0)(\gamma|\alpha)_{n-k}.
               \end{equation}   
               \end{proof}
               
                 \begin{theorem}For $\alpha,\beta,\gamma,x\in\mathbb{N}_0$ such that $(\alpha,\beta,\gamma,x)\not=(0,0,0,0)$,\label{theorem:3}
                 \fontsize{10}{14}
                 \begin{equation}
                 T_{n+1}^{1,x}(\alpha,\beta,\gamma)=\gamma T_n^{1,x}(\alpha,\beta,\gamma-\alpha)+x\beta\sum\limits_{k=0}^{n}\binom{n}{k}T_{n-k}^{1,x}(\alpha,\beta,\gamma)T_{k}^{1,x}(\alpha,\beta,\beta-\alpha).\end{equation}
                 \end{theorem}
                 \begin{proof}
                 The argument is based on the position of the $(n+1)th$ element.  
                 
   If the $(n+1)th$ element is on the special section (with $\gamma$ compartments), then we may choose in $\gamma$ ways a compartment for the element. There are $\gamma-\alpha$ free compartment left for the other elements to occupy within this section. The remainder $n$ elements can be arranged among the two sections in $T_n^{1,x}(\alpha, \beta, \gamma- \alpha)$ ways. 
                 
    Next, we consider the case when $(n+1)$th goes into the other section with Property~\ref{property:1} and let $B$ be the block that includes it. Further, let $k$ be the number of elements that are included in the block $B$ and blocks arranged right to $B$ within the section. This part of the preferential arrangement can be viewed as one on $k$ elements having a special block $B$ with $\beta-\alpha$ compartments. The other part of the preferential arrangement, the blocks that go to the left of $B$ including those in the special section with $\gamma$ compartments, can be viewed simply as a preferential arrangement on the remainder $n-k$ elements. Hence, we choose the $k$ elements in $\binom{n}{k}$ ways, select a compartment for the $(n+1)$th element in $\beta$ ways, a color for the block $B$ in $x$ ways, and construct the preferential arrangements on the $k$ element in $T_k^{1,x}(\alpha,\beta,\gamma+\beta-\alpha)$ ways, and on the remainder $n-k$ elements in $T_{n-k}^{1,x}(\alpha,\beta,\gamma)$ ways.
                 
                 \end{proof}
                 
                  Theorem~\ref{theorem:3} is a generalisation of Theorem 8 of \cite{Our paper on Generalised barred preferential arrangements}.
                 
                 \begin{theorem}For $\alpha,\beta,\gamma,x\in\mathbb{N}_0$ such that $(\alpha,\beta,\gamma,x)\not=(0,0,0,0)$,\fontsize{10}{14}
                 \begin{equation}
                 T_{n+1}^{1,x}(\alpha,\beta,\gamma)=\gamma T_n^{1,x}(\alpha,\beta,\gamma-\alpha)+x\beta\sum\limits_{k=0}^{n}\binom{n}{k}T_k^{1,x}(\alpha,\beta,\gamma+\beta-\alpha)T_{n-k}^{1,x}(\alpha,\beta,0).
                 \end{equation}
                 \end{theorem}
                 
                \begin{proof}The proof is similar to the previous theorem. For the case the $(n+1)th$ element is on the special section having $\gamma$ compartments is as in the previous theorem. 
                
                We now consider the case where the $(n+1)th$ element is in the section having property~\ref{property:1}, say the element is part of block $B$. Now consider the block $B$, and the  special section having $\gamma$ compartments as a single unit. Let $k$ be the number of elements on the left to the block $B$, within the section. This unit having $\beta+\gamma-\alpha$ available compartments.
                
                Given $k$ elements, there are $T_k^{1,x}(\alpha,\beta,\gamma+\beta-\alpha)$ possibilities that the elements are on either the unit or on the left of $B$ within the same section  that $B$ is on. The remainder $n-k$ elements may be preferential arrangements to the right of $B$ in  $T_{n-k}^{1,x}(\alpha,\beta,0)$ ways. 
                \end{proof}

\end{section}

\section{When multiple bars are considered}

In this section we examine the numbers $T_n^{\lambda,x}(\alpha,\beta,\gamma)$, for an arbitrary $\lambda$.  

\begin{theorem}
\label{theorem:5}
 The number $T_n^{\lambda,x}(\alpha,\beta,\gamma)$  for $\alpha,\beta,\gamma,\lambda$ in non-negative integers where $\alpha|\beta$ and $\alpha|\gamma$, is the number of barred preferential arrangements having $\lambda$ bars, such that one section has property~\ref{property:2} and  $\lambda$ other sections have property~\ref{property:1}.

\end{theorem}

\begin{proof}
 By \eqref{equation:1} we have 
 \fontsize{10}{14}
 \begin{equation}\label{equation:14}T_n^{\lambda,x}(\alpha,\beta,\gamma)=\sum\limits_{r_1+\cdots+r_{\lambda+1}}\binom{n}{r_1,\ldots,r_{\lambda+1}}(\gamma|\alpha)_{r_1}\prod\limits_{i=2}^{\lambda+1}G^x_{r_i}(\alpha,\beta,0)\end{equation}
\end{proof}

For the case $\lambda=1$ the statement of Theorem~\ref{theorem:5} can be derive from the one given for the numbers $B_n(\alpha,\beta,\gamma)$ in \cite{On generalised Bell polynomials}.

 \begin{corollary}\label{corolloary:1}
                                           For $\lambda,\beta,\gamma, x\in\mathbb{N}_0$,
                                           $$\frac{e^{\gamma t}}{[1-x(e^{\beta t}-1)]^\lambda},$$
 is the generating function for the number of barred preferential arrangements, having $\lambda$ bars. Where elements on each block on fixed $\lambda$ sections are colored with $\beta$ available colors, and elements on one section are colored with $\gamma$ available colors. Also the blocks themselves on the first $\lambda$ sections being colored with one of $x$ available colors. This is a generalisation of the work done in \cite{Our paper on Generalised barred preferential arrangements}.                                       \end{corollary}

                           \begin{remark}
                           [Geometric Polynomials]
    Geometric polynomials go far back as Euler's work on the year 1755 (see page 389 on Part \RNum{2} of \cite{Euler's work on geometric polynomials}). These polynomials are 
    well known in the literature, and it is also well known that these polynomials arise from variations of the  generating function $\frac{e^{r t}}{(1-x(e^{t}-1))^r}$, for instance in \cite{On generalised Bell polynomials,A series transformation formula,Apostol-Bernoulli functions,close encounters with stirling numbers of the second kind,Investing Geometric and exponential polynomials,Geometric Polynomials: Properties and Applications to Series with Zeta Values,Polynomials Related to Harmonic Numbers and Evaluation of Harmonic Number Series II,Fourier series of sums of products of ordered Bell and poly-Bernoulli functions,An explicit formula for Bernoulli polynomials in terms of $ r $-Stirling numbers of the second kind,Generalization of Mellin derivative and its applications}. 
                         Hence, our results in Theorem~\ref{theorem:5} offers a generalised combinatorial interpretation of these geometric polynomials. 
                         .     
                           \end{remark} 
               
          \begin{remark}[Nelsen-Schmidt question] Both generation functions, $\frac{e^{\gamma t}}{(2-e^t)^{\lambda}}$ and $\frac{e^{\gamma t}}{(2-e^{\beta t})^{\lambda}}$ previously studied in answering the Nelsen-Schmidt question considered by the authors in  \cite{Nkonkobe & Murali Nelesn-Schmidt generating function,Our paper on Generalised barred preferential arrangements} arising as special cases of the generating function we study here given in \eqref{equation:1}.
                                                                       \end{remark}

\begin{theorem}\label{theorem:7}For $\alpha,\lambda,\beta,x\in\mathbb{N}_0$ such that $(\alpha,\beta,\gamma,x)\not=(0,0,0,0)$,\fontsize{10}{14}
\begin{equation}
T_{n+1}^{\lambda,x}(\alpha,\beta,\gamma)=\gamma T_n^{\lambda,x}(\alpha,\beta,\gamma-\alpha)+x\lambda\beta T_n^{\lambda+1,x}(\alpha,\beta,\gamma+\beta-\alpha).
\end{equation}
\end{theorem}
\begin{proof}

Consider the $(n+1)$th element. If it is in the special block, we have to select a compartment in $\gamma$ ways and to arrange the remainder $n$ elements in $ T_n^{\lambda,x}(\alpha,\beta,\gamma-\alpha)$ ways.

In the other case the $(n+1)th$ element is in one of the sections with property~\ref{property:1}.
Let $B$ be the block with $\beta$ compartments which the $(n+1)$th element forms part-off. Consider the first special block and $B$ as united special block. However, to be able to reconstruct the original situation, mark the place of $B$, with an extra bar. Viewing our object this way, we have a preferential arrangement on $n$ elements having a special block with $\gamma+\beta-\alpha$ (after dropping the $(n+1)$th element, $\alpha$ compartments are closed from $\beta$ compartments of $B$) with $\lambda +1$ sections having property~\ref{property:1}. 
\end{proof}

Theorem~\ref{theorem:7} is a generalisation of Theorem 9 of \cite{Our paper on Generalised barred preferential arrangements}.
\begin{theorem}For $\alpha,\beta,\gamma,\lambda,x\in\mathbb{N}_0$ such that
$(\alpha,\beta,\gamma,\lambda,x)\not=(0,0,0,0,0)$,\fontsize{10}{14}
\begin{equation}
T_{n+1}^{\lambda,x}(\alpha,\beta,\gamma)=\gamma T_n^{\lambda,x}(\alpha,\beta,\gamma-\alpha)+x\lambda\beta\sum\limits_{s=0}^{n}\binom{n}{s}(\gamma|\alpha)_sT_{n-s}^{\lambda+1,x}(\alpha,\beta,\beta-\alpha).
\end{equation}

\end{theorem}
\begin{proof}
We consider again the position of the $(n+1)$th element. If it is in the first special block, we have $\gamma T_n^{\lambda,x}(\alpha,\beta,\gamma-\alpha)$ possibilities. Now, assume it is in one of the sections where each block has $\beta$ compartments. 
We let $B$ denote the block of which the $(n+1)$th element is part-off. Clearly, we need to select the color for the block $B$ in $x$ ways, the compartment for $(n+1)$th element in $\beta$ ways and the section of the block in $\lambda$ ways. We build the special block from $s$ chosen element (in $\binom{n}{s}(\gamma|\alpha)_s$ ways). Now treating $B$ as a special block with $\beta-\alpha$ available compartments, the remainder $n-s$ elements can be arranged $T^{\lambda+1,x}_{n-s}(\alpha,\beta,\beta-\alpha)$ ways, this is from the fact that within the section in which the $(n+1)th$ element is in, each of the left side and the right side of the block $B$ excluding $B$ gives rise to a single section having property~\ref{property:1}.
\end{proof}

\begin{theorem}For $\alpha,\beta,\gamma,\lambda,x\in\mathbb{N}_0$ such that $(\alpha,\beta,\gamma,\lambda,x)\not=(0,0,0,0,0)$,
\begin{equation}
(1+x)T_n^{\lambda+1,x}(\alpha,\beta,0)= T_n^{\lambda,x}(\alpha,\beta,0)+ xT_n^{\lambda+1,x}(\alpha,\beta,\beta).
\end{equation}
\end{theorem}
\begin{proof}
We reorder the equation in the following way, 
\begin{equation}\label{equation:7}
xT_n^{\lambda+1,x}(\alpha,\beta,\beta)-xT_n^{\lambda+1,x}(\alpha,\beta,0)=T_n^{\lambda+1,x}(\alpha,\beta,0)-T_n^{\lambda,x}(\alpha,\beta,0).
\end{equation}

The left hand side of \eqref{equation:7} counts the preferential arrangements having a non-empty special block with $\beta$ compartments, which is colored by one of the $x$ colors. Since the color is also chosen in the case when the first block is empty, one can think of that as a marked block color.
On the other hand, from a preferential arrangement with $\lambda+1$ sections and empty first section, we obtain a preferential arrangement with a non-empty first section, if we shift the first bar to the right, directly after the first block. This block has a color (out of $x$) originally. However, if there are no blocks directly to the righ of the first bar, this can not be done, we need to reduce by the number of these preferential arrangements, which is exactly the number of the same preferential arrangements with one less bar. 
\end{proof}
\begin{theorem}For $\alpha,\beta,\gamma,\lambda,x\in\mathbb{N}_0$ such that
 $(\alpha,\beta,\gamma,\lambda,x)\not=(0,0,0,0,0)$,\label{theorem:9}
\fontsize{10}{14}
\begin{equation}
T_{n+1}^{\lambda,x}(\alpha,\beta,\gamma)=\gamma T_n^{\lambda,x}(\alpha,\beta,\gamma-\alpha)+x\beta \lambda\sum\limits_{k=0}^{n}\binom{n}{k}T_k^{\lambda,x}(\alpha,\beta,\gamma)T_{n-k}^{1,x}(\alpha,\beta,\beta-\alpha).\end{equation}
\end{theorem}

\begin{proof}
Consider the $(n+1)$th element. The proof goes similarly to the previous ones. Let $B$ denote the block containing the $(n+1)$th element, and let this block fall between the bars $|^*$(to the left of $B$) and $|^{**}$(to the right of $B$). The block $B$ can be interpreted as a special block for those elements that are either part of the block $B$ or arranged to the right of $B$ before the bar $|^{**}$, $k$ elements can be arranged here in $T_{k}^{1,x}(\alpha,\beta,\beta-\alpha)$ ways. The remaining $n-k$ elements can be arranged on the other sections in $T^{\lambda,x}_{n-k}$ ways, where the extra section is the section to the left of $B$ excluding $B$ before the bar $|^*$. Clearly, we need to choose the $k$ elements out of the $n$ in $\binom{n}{k}$ ways, the section in $\lambda$ ways, the compartment in $\beta $ ways and the color of the block $B$ in $x$ ways.
\end{proof}

Theorem~\ref{theorem:9} is a generalisation of Theorem 8 of \cite{Our paper on Generalised barred preferential arrangements}.
\begin{theorem}For $\alpha,\beta,\gamma,\lambda,x\in\mathbb{N}_0$ such that
 $(\alpha,\beta,\gamma,\lambda,x)\not=(0,0,0,0,0)$,
\fontsize{10}{14}
\begin{equation}
T_{n+1}^{\lambda,x}(\alpha,\beta,\gamma)=\gamma\sum\limits_{k=0}^{n}\binom{n}{k}(\gamma-\alpha|\alpha)_kT_{n-k}^{\lambda,x}(\alpha,\beta,0)+x\lambda \beta\sum\limits_{k=0}^{n}\binom{n}{k}(\gamma+\beta-\alpha|\alpha)_k T_{n-k}^{\lambda+1,x}(\alpha,\beta,0).
\end{equation}
\end{theorem}
\begin{proof}
We locate the position of the $(n+1){th}$ element. 

When the $(n+1){th}$ element is on the special section having $\gamma$ compartments, there are $\gamma$ ways of choosing a compartment for the element. Off the $n$ other elements, $k$ of them can also form part of the special section in $(\gamma-\alpha|\alpha)_k$ ways. The remaining $n-k$ elements can be arranged on the other sections with property~\ref{property:1} in $T^{\lambda,x}_{n-k}(\alpha,\beta,0)$ ways. 

In the second case the $(n+1)th$ element is in one of the sections with property~\ref{property:1}.  We denote the block of which the $(n+1){th}$ element is part of by $B$. Clearly, a section, compartment and color for the block $B$ can be chosen in $x\lambda\beta$ ways. Now treating the block $B$ and the special block having $\gamma$ compartments as a single unit, $k$ elements can be arranged in this unit in $(\gamma+\beta-\alpha)_k$ ways. The remaining $n-k$ elements can be arranged on the other places in $T^{\lambda+1,x}_{n-k}(\alpha,\beta,0)$ ways. 
\end{proof}

\begin{theorem}\label{theorem:12}For $\alpha,\beta,\gamma,\lambda,x\in\mathbb{N}_0$ such that
$(\alpha,\beta,\gamma,\lambda,x)\not=(0,0,0,0,0),$
\fontsize{10}{14}
\begin{equation}
T_{n+1}^{\lambda,x}(\alpha,\beta,\gamma)=\gamma T_n^{\lambda,x}(\alpha,\beta,\gamma-\alpha)+\sum\limits_{k=0}^{n}\binom{n}{k}(\gamma|\alpha)_kT_{n-k+1}^{\lambda,x}(\alpha,\beta,0).
\end{equation}
\end{theorem}
\begin{proof}

Similar to the previous theorems, we locate the position of the $(n+1)^{th}$ element. 

The $(n+1){th}$ element is either on the special block or in one of the sections with property~\ref{property:1}. When the $(n+1){th}$ element is not on the special section, $k$ elements can be chosen and arranged on the special section in $(\gamma|\alpha)_k$ ways. The $n-k$ other elements plus the $(n+1)th$ element, can be arranged on the other sections having property~\ref{property:1} in $T^{\lambda,x}_{n-k+1}(\alpha,\beta,0)$ ways.
\end{proof}

           In \cite{Higher order geometric polynomials} the authors proposed the following identity, \\             $T_n^{\lambda+1,x}(\alpha,\beta,\gamma)=\sum\limits_{k=0}^{n}\binom{k+\lambda}{k}k!S(n,k,\alpha,\beta,\gamma)
                       \beta^kx^k$. In the following theorem we give a combinatorial interpretation of the result.

           \begin{theorem}\label{theorem:8}For $\alpha,\beta,\gamma,\lambda,x\in\mathbb{N}_0$ such that
           $(\alpha,\beta,\gamma,\lambda,x)\not=(0,0,0,0,0),$\fontsize{10}{1}
           \begin{equation}\label{equation:18}
           T_n^{\lambda,x}(\alpha,\beta,\gamma)=\sum\limits_{k=0}^{n}\binom{k+\lambda-1}{k}k!S(n,k,\alpha,\beta,\gamma)
           \beta^kx^k.
           \end{equation}
           \end{theorem} 
           \begin{proof}
            In this theorem we use the notion of forming a barred preferential arrangement by first forming an ordered set partition and then insert bars. The $n$ elements can be distributed into $k+1$ cells where the first $k$ cells have property~\ref{property:1} and the $(k+1)th$ cell has property~\ref{property:2}, such that the first $k$ cells are non-empty in $\beta^kk!S(n,k,\alpha,\beta,\gamma)$ ways (see Theorem~\ref{Theorem:15}). The first $k$ cells can be coloured in $x^k$ ways. Then, $\lambda-1$ bars can be inserted in-between the first $k$ cells having property~\ref{property:1} in $\binom{k+\lambda-1}{k}$.  
           \end{proof}
           
           Theorem~\ref{theorem:8} is a generalisation of theorem 3 of \cite{barred:2013}.

              
              \vspace{4mm}
 \noindent Combining \eqref{equation:18} and the following recurrence relation from \cite{A unified approach to generalized Stirling numbers},
               \fontsize{10}{14}
          \begin{equation}\label{equation:17}
   S(n+1,k,\alpha,\beta,\gamma)=S(n,k-1,\alpha,\beta,\gamma)+(k\beta-n\alpha+\gamma)S(n,k,\alpha,\beta,\gamma),
\end{equation}\normalsize
we obtain \eqref{equation:19} below. In Theorem~\ref{theorem:11} we give a combinatorial interpretation of the result.
        
  \begin{theorem}\label{theorem:11} For $\lambda\geq 1$, and $\alpha,\beta,\gamma,\lambda,x\in\mathbb{N}_0$ such that
  $(\alpha,\beta,\gamma,\lambda,x)\not=(0,0,0,0,0),$
   \fontsize{10}{1}
          \begin{equation}\label{equation:19}
          T_{n+1}^{\lambda+1,x}(\alpha,\beta,\gamma)=(\gamma|\alpha)_{n+1}+\sum\limits_{k=1}^{n}\binom{\lambda+k}{k}[S(n,k-1,\alpha,\beta,\gamma)+(k\beta-n\alpha+\gamma)S(n,k,\alpha,\beta,\gamma)]\beta^kk!x^k.
          \end{equation}
          \end{theorem}
         \begin{proof}
         We divide the proof into three cases. 
         
         Case 1: all the elements are on the special section, there are $(\gamma|\alpha)_{n+1}$ ways of arranging the elements.
         
         Case 2: the $(n+1)th$ element is on its own block on one of the sections having property 1.
         Of the $n$ elements $k$ blocks can be formed in $\beta^{k-1}(k-1)!S(n,k-1,\alpha,\beta,\gamma)x^{k-1}$ ways where only the $(k)th$  block is allowed to be empty, satisfying the conditions of Theorem~\ref{theorem:10}. The $(n+1)th$ element can be put to form its own block on the $k$ spaces in-between the $k-1$ blocks that are not allowed to be empty and a color of this block in $k\beta x$ ways. Now $\lambda$ bars can be inserted in-between the spaces of the $k$ non-empty blocks in $\binom{k+\lambda}{k}$. Hence, the total number of possibilities in this case is $\sum\limits_{k=1}^n\binom{k+\lambda}{k}k!\beta^kS(n,k-1,\alpha,\beta,\gamma)x^k$.
         
         Case 3: the $(n+1)th$ element is either on the special block or on a block with other elements. 
        By Theorem~\ref{theorem:10} the $n$ other elements can form $k+1$ blocks in $\beta^kk!S(n,k,\alpha,\beta,\gamma)$ ways, where only the special block with $\gamma$ compartments is allowed to be empty. The non-empty blocks can be colored in $x^k$ ways. As the last element to be placed the $(n+1)th$ element can be placed on one of these $k+1$ blocks in $k\beta+\gamma-\alpha n$ ways. Then, $\lambda$ bars can be placed in between the non-empty blocks in $\binom{k+\lambda}{k}$ ways. Hence, the total number of possibilities in this case is $\sum\limits_{k=1}^n\binom{k+\lambda}{k}\beta^kk!(k\beta+\gamma-\alpha n)S(n,k,\alpha,\beta,\gamma)x^k$.

         \end{proof}


       \begin{section}{Asymptotic Analysis}
       
        In this section, an asymptotic expansion for the higher order geometric polynomials will be derived using the known result of Hsu \cite{Power-type} on asymptotic expansion formula for the coefficients of power-type generating functions involving large parameters.
       
       \smallskip
       Let $\mathbb{N}$ be the set of positive integers and $\sigma(n)$ be the set of partitions of integer $n\in\mathbb{N}$, represented by $1^{k_1}2^{k_2}\ldots n^{k_n}$ such that
       $$1k_1+2k_2+\ldots +nk_n=n, k_i\geq0 \;\;(i=1,2,\ldots, n)\;\; \& \;\;k=k_1+k_2+\ldots+k_n.$$
       The parameter $k$ denotes the numbers of parts of the partition. Let $\sigma(n,k)$ be the subset of $\sigma(n)$ consisting of partitions of $n$ with $k$ parts.
       
       \smallskip
       Consider a formal power series $\phi(t)=\sum\limits_{n=0}^{\infty}a_nt^n$ over the complex field $\mathbb{C}$ with $a_0=\phi(0)=1$. For every $j$ $(0\leq j<n)$, let $W(n,j)$ be equal to the following sum,
       \begin{equation}\label{eq1}
       W(n,j)=\sum_{\sigma(n,n-j)}\frac{a_1^{k_1}a_2^{k_2}\ldots a_n^{k_n}}{k_1!k_2!\ldots k_n!}.
       \end{equation}
       Using the notation $\left[t^n\right]\left(\phi(t)\right)^{\lambda}$ for the coefficient of $t^n$ in the power series expansion of $\left(\phi(t)\right)^{\lambda}$, 
       \begin{equation}\label{eq2}
       \frac{1}{(\lambda)_n}\left[t^n\right]\left(\phi(t)\right)^{\lambda}=\sum_{j=0}^s\frac{W(n,j)}{(\lambda-n+j)_j}\;\;+o\left(\frac{W(n,s)}{(\lambda-n+s)_s}\right).
       \end{equation}
       To derive the said asymptotic formula, we need to consider first the generating function in Lemma~\ref{lemma:11}. We use $\phi(t)$ to denote this generating function. That is,
       \begin{equation}\label{eq3}
       \phi(t) = \sum_{n=0}^{\infty} T_{n}^{1,x} (\alpha,\beta,\gamma) \frac{t^n}{n!} = \frac{(1+\alpha t)^{\frac{\gamma}{\alpha}}}{\left(1-x\left((1+\alpha t)^{\frac{\beta}{\alpha}}-1\right)\right)}. 
       \end{equation}
       Then,
       \begin{equation*}
       (\phi(t))^{\lambda} = \frac{(1+\alpha t)^{\frac{\lambda\gamma}{\alpha}}}{\left(1-x\left((1+\alpha t)^{\frac{\beta}{\alpha}}-1\right)\right)^{\lambda}} = \sum_{n=0}^{\infty} T_{n}^{\lambda,x} (\alpha,\beta,\gamma) \frac{t^n}{n!}.
       \end{equation*}
       By making use of \eqref{eq2}, we have
       \begin{equation*}
       \frac{T_{n}^{\lambda,x} (\alpha,\beta,\gamma)}{(\lambda)_nn!} = \sum_{j=0}^{s} \frac{W(n,j)}{(\lambda-n+j)_j} + o\left(\frac{W(n,s)}{(\lambda-n+s)_s}\right),
       \end{equation*}
       where $n=o\left(\lambda^{1/2}\right)(\lambda\to\infty)$ and the numbers $W(n,j)$ are given in \eqref{eq1} with $a_j$ being determined by \eqref{eq3}, namely,
       \begin{align*}
       a_j &= [t^j]\phi(t) = \frac{T_{j}^{1,x} (\alpha, \beta, \gamma)}{j!} \\
           &= [t^j]\frac{(1+\alpha t)^{\frac{\gamma}{\alpha}}}{\left(1-x\left((1+\alpha t)^{\frac{\beta}{\alpha}}-1\right)\right)}.
       \end{align*}
       Note that 
       \begin{align*}
       \sum_{n=0}^{\infty} T_{n}^{1,x} (\alpha,\beta,\gamma) \frac{t^n}{n!}&=
       \frac{(1+\alpha t)^{\frac{\gamma}{\alpha}}}{\left(1-x\left((1+\alpha t)^{\frac{\beta}{\alpha}}-1\right)\right)} \\
       		&= (1+\alpha t)^{\frac{\gamma}{\alpha}} \left(1-x\left((1+\alpha t)^{\frac{\beta}{\alpha}}-1\right)\right)^{-1} \\
       		&= (1+\alpha t)^{\frac{\gamma}{\alpha}} \sum_{k=0}^{\infty} x^k \left[(1+\alpha t)^{\frac{\beta}{\alpha}}-1\right]^k \\  
       		&= (1+\alpha t)^{\frac{\gamma}{\alpha}} \sum_{k=0}^{\infty} x^k \left[\left((1+\alpha t)^{\frac{\beta}{\alpha}}-1\right)^k\right] \\
       		&= (1+\alpha t)^{\frac{\gamma}{\alpha}} \sum_{k=0}^{\infty} x^k \sum_{j=0}^{k} \paren{k}{j} (1+\alpha t)^{(\frac{\beta}{\alpha})j} (-1)^{k-j} \\
       		&= \sum_{k=0}^{\infty} x^k \sum_{j=0}^{k} \paren{k}{j} (1+\alpha t)^{(\frac{\beta}{\alpha})j + \frac{\gamma}{\alpha}} (-1)^{k-j} \\
       		&= \sum_{k=0}^{\infty} \sum_{j=0}^{k} x^k \paren{k}{j} (-1)^{k-j} (1+\alpha t)^{(\frac{\beta}{\alpha})j + \frac{\gamma}{\alpha}} \\
       		&= \sum_{k=0}^{\infty} \sum_{j=0}^{k} x^k \paren{k}{j} (-1)^{k-j} \sum_{i=0}^{\infty} \paren{\frac{\beta j+\gamma}{\alpha}}{i} (\alpha t)^i \\
       		\end{align*}
       
       So, \begin{align*}
       \sum_{n=0}^{\infty} T_{n}^{1,x} (\alpha,\beta,\gamma) \frac{t^n}{n!}&= \sum_{i=0}^{\infty}\sum_{k=0}^{\infty}\sum_{j=0}^{k} x^k \paren{k}{j} (-1)^{k-j} \paren{\frac{\beta j+\gamma}{\alpha}}{i} (\alpha t)^i \\
              		&= \sum_{i=0}^{\infty}\sum_{k=0}^{\infty}\sum_{j=0}^{k} x^k \paren{k}{j} (-1)^{k-j} \paren{\frac{\beta j+\gamma}{\alpha}}{i} \alpha^i t^i.
       \end{align*}
       Comparing coefficients,
       \begin{equation*}
       \frac{T_{i}^{1,x} (\alpha,\beta,\gamma)}{i!} = \sum_{k=0}^{\infty}\sum_{j=0}^{k} x^k \paren{k}{j} (-1)^{k-j} \paren{\frac{\beta j+\gamma}{\alpha}}{i} \alpha^i
       \end{equation*}
       or
       \begin{align*}
       T_{n}^{1,x} (\alpha,\beta,\gamma) &= \sum_{k=0}^{\infty}\sum_{j=0}^{k} n! x^k \paren{k}{j} (-1)^{k-j} \paren{\frac{\beta j + \gamma}{\alpha}}{n} \alpha^n \\ 
       &= n!\alpha^n \sum_{k=0}^{\infty}\sum_{j=0}^{k} x^k \paren{k}{j} (-1)^{k-j} \paren{\frac{\beta j + \gamma}{\alpha}}{n} \\
       &= n!\alpha^n \sum_{k=0}^{\infty} \left\lbrace\sum_{j=0}^{k} \paren{k}{j} (-1)^{k-j} \paren{\frac{j\beta + \gamma}{\alpha}}{n}\right\rbrace x^k \\
       &= \frac{n!}{n!}\alpha^n \sum_{k=0}^{\infty} \left\lbrace \sum_{j=0}^{k} \paren{k}{j} (-1)^{k-j} \left(\frac{j\beta + \gamma}{\alpha}\right)_n \right\rbrace x^k \\
       &= \alpha^n \sum_{k=0}^{\infty} \left\lbrace \sum_{j=0}^{k} \paren{k}{j} (-1)^{k-j} \frac{1}{\alpha^n} (j\beta + \gamma | \alpha)_n \right\rbrace x^k.
       \end{align*}
       Using the explicit formula for the unified generalization of Stirling numbers \cite{Corcino Rorberto 2001}, we have
       \begin{equation*}
       T_{n}^{1,x} (\alpha,\beta,\gamma)= \sum_{k=0}^{\infty} \beta^k k! S(n,k;\alpha,\beta,\gamma)x^k.
       \end{equation*}
       By Lemma~\ref{lemma:10}, we obtain
       \begin{equation*}
       T_{n}^{1,x} (\alpha,\beta,\gamma) =  G_n^x(\alpha, \beta, \gamma).
       \end{equation*} 
       When $x=1$, these polynomials yield the numbers in \cite{On generalised Bell polynomials}, denoted by $B(n;\alpha, \beta, \gamma)$. That is,
       \begin{equation*}
       T_{n}^{1,1} (\alpha,\beta,\gamma) =  B_n(\alpha, \beta, \gamma).
       \end{equation*} 
       The following theorem formally states the above asymptotic formula.
      
       \begin{theorem}
       There holds the asymptotic formula
       \begin{equation*}
       \frac{T_{n}^{\lambda,x} (\alpha,\beta,\gamma)}{(\lambda)_nn!} = \sum_{j=0}^{s} \frac{W(n,j)}{(\lambda-n+j)_j} + o\left(\frac{W(n,s)}{(\lambda-n+s)_s}\right),
       \end{equation*}
       for $\lambda\to\infty$ with $n=o\left(\lambda^{1/2}\right)$, where the numbers $W(n,j)$ are defined in \eqref{eq1} with $a_j$ being given by
       $$a_j=\frac{1}{j!}G_j^x(\alpha, \beta, \gamma).$$
       \end{theorem} 
       
       \smallskip
       Assume $\gamma=0$ and take $s=2$. Notice that the computation of $W(n,j)$ is based on the number of partitions $\sigma(n,n-j)$ of $n$ with $k=n-j$ parts (cf. Hsu and Shiue \cite{A unified approach to generalized Stirling numbers}).  So, when $j=0$, we need to compute $\sigma(n,n)$, the number of partion of $n$ with $n$ parts. That is, finding $k_i$'s satisfying
       $$k=n = k_1 + k_2 + \cdots + k_n=k_1\ \mbox{and}\ n=k_1 + 2k_2 + \cdots + nk_n = k_1.$$
       Hence, we have
       \begin{equation*}
       W(n,0) = \sum_{\sigma(n,n)} \frac{a_{1}^{k_1}a_{2}^{k_2} \cdots a_{n}^{k_n}}{k_{1}!k_{2}! \cdots k_{n}!} = \frac{1}{n!} a_{1}^n.
       \end{equation*}
       For $j=1$, we need to compute $\sigma(n,n-1)$. That is, finding $k_i$'s satisfying
       \begin{align*}
       k&=n-1 = k_1 + k_2 + \cdots + k_n=k_1+k_2=(n-2)+1\\
       n& = k_1 + 2k_2 + ... + nk_n=(n-2)+2(1).
       \end{align*}
       That is, $k_1=n-2$ and $k_2=1$. Hence, we have
       \begin{equation*}
       W(n,1) = \sum_{\sigma(n,n-1)} \frac{a_{1}^{k_1}a_{2}^{k_2} \cdots a_{n}^{k_n}}{k_{1}!k_{2}! \cdots k_{n}!} = \frac{1}{(n-2)!} a_{1}^{n-2} a_2 
       \end{equation*}
       Now, for $j=2$, 
       \begin{align*}
       k&=n-2 = k_1 + k_2 + \cdots + k_n=k_1+k_2=(n-3)+1\\
       n&= k_1 + 2k_2 + 3k_3+\ldots+ nk_n=(n-3)+3(1).
       \end{align*}
       and
       \begin{align*}
       k&=n-2 = k_1 + k_2 + \cdots + k_n=k_1+k_2=(n-4)+2\\
       n&= k_1 + 2k_2 + 3k_3+\ldots + nk_n=(n-4)+2(1).
       \end{align*}
       Hence, the sum is composed of terms: one term is using the index 
       $$(k_1, k_2, k_3, k_4, \cdots , k_n)=(n-3, 0, 1, 0, \cdots , 0)$$
       and the other term is using the index
       $$(k_1, k_2, k_3, k_4, \cdots , k_n)=(n-4, 2, 0, 0, \cdots , 0).$$
       Hence, we have
       \begin{align*}
       W(n,2) &= \sum_{\sigma(n,n-2)} \frac{a_{1}^{k_1}a_{2}^{k_2} \cdots a_{n}^{k_n}}{k_{1}!k_{2}! \cdots k_{n}!}\\
       &= \frac{1}{(n-3)!} a_{1}^{n-3} a_3 + \frac{1}{2!(n-4)!} a_{1}^{n-4} a_{2}^{2} .
       \end{align*}
       Now, we can compute the approximate value of $T_{n}^{\lambda,x} (\alpha,\beta,0)$ as follows:
       \begin{equation*}
       \frac{T_{n}^{\lambda,x} (\alpha,\beta,0)}{n!} \sim (\lambda)_n W(n,0) + (\lambda)_{(n-1)} W(n,1) + (\lambda)_{(n-2)} W(n,2)
       \end{equation*}
       where $n = o(\lambda^{1/2})$ as $\lambda \rightarrow \infty$. With
       \begin{align*}
       W(n,0) &= \frac{1}{n!}a_{1}^{n} = \frac{1}{n!} \left( \sum_{k=0}^1 \beta^k k! S(1,k;\alpha,\beta,0)x^k \right)^n \\
       &= \frac{1}{n!} \left\lbrace S(1,0;\alpha,\beta,0) + \beta S(1,1;\alpha,\beta,0)x \right\rbrace^n \\
       &= \frac{1}{n!} \lbrace 0 + \beta x \rbrace^n = \frac{\beta^n x^n}{n!} \\
       W(n,1) &= \frac{1}{(n-2)!} a_{1}^{n-2} a_2, 
       \end{align*}
       where $a_1 = \beta x$ and
       \begin{align*}
       a_2 &= T_{2}^{1,x} (\alpha,\beta,0) = \sum_{k=0}^2 \beta^k k! S(2,k;\alpha,\beta,0)x^k \\
       		&= 0 + \beta S(2,1;\alpha,\beta,0)x + \beta^2(2!) S(2,2;\alpha,\beta,0)x^2 \\
       		&= \beta S(2,1;\alpha,\beta,0)x + 2\beta^2(1)x^2 \\
       		&= \beta(\beta - \alpha)x + 2\beta^2 x^2
       \end{align*}
       \begin{equation*}
       W(n,1) = \frac{1}{(n-2)!} \lbrace (\beta x)^{n-2} \left( \beta(\beta - \alpha)x + 2\beta^2 x^2 \right) \rbrace 
       \end{equation*}
       \begin{equation*}
       W(n,2) = \frac{1}{(n-3)!} a_{1}^{n-3} a_3 + \frac{1}{2!(n-4)!} a_{1}^{n-4} a_{2}^{2},
       \end{equation*}
       where $a_1 = \beta x$, $a_2 = \beta(\beta - \alpha)x + 2\beta^2 x^2$, and
       \begin{align*}
       a_3 &= T_{3}^{1,x} (\alpha,\beta,0) \\
       		&= \sum_{k=0}^{3} \beta^k k! S(3,k;\alpha,\beta,0)x^k \\
       		&= S(3,0;\alpha,\beta,0) + \beta S(3,1;\alpha,\beta,0)x\\ 
       		&\;\;\;+ 2\beta^2 S(3,2;\alpha,\beta,0)x^2 +3!\beta^3 S(3,3;\alpha,\beta,0)x^3,
       \end{align*}
       where
       \begin{align*}
       S(3,0;\alpha,\beta,0) &= 0 \\
       S(3,1;\alpha,\beta,0) &= S(2,0;\alpha,\beta,0) + (\beta - 2\alpha)S(2,1;\alpha,\beta,0)\\
       &= (\beta - 2\alpha)(\beta - \alpha) \\
       S(3,2;\alpha,\beta,0) &= S(2,1;\alpha,\beta,0) + (2\beta - 2\alpha)S(2,2;\alpha,\beta,0) \\
       &= (\beta - \alpha) + 2\beta - 2\alpha = 3\beta - 3\alpha \\
       S(3,3;\alpha,\beta,0) &= 1. 
       \end{align*}
       Hence,
       \begin{equation*}
       a_3 = \beta(\beta -  \alpha)(\beta - 2\alpha)x + 2\beta^2(3\beta - 3\alpha)x^2 + 6\beta^3 x^3
       \end{equation*}
       \begin{multline*}
       W(n,2) = \frac{1}{(n-3)!} (\beta x)^{n-3} \left\lbrace (\beta | \alpha)_3 x + 6\beta^2(\beta - \alpha)x^2 + 6\beta^3 x^3 \right\rbrace \\ 
       					+ \frac{1}{2!(n-4)!} (\beta x)^{n-4} \left\lbrace \beta(\beta - \alpha)x + 2\beta^2 x^2 \right\rbrace^2
       \end{multline*}
       
       Consequently,
       \begin{equation*}
       \begin{split}
       \frac{T_{n}^{\lambda,x} (\alpha,\beta,0)}{n!} &\sim (\lambda)_n \frac{\beta^n x^n}{n!} + (\lambda)_{n-1} \frac{1}{(n-2)!} (\beta x)^{n-2} \left\lbrace \beta(\beta - \alpha)x + 2\beta^2 x^2 \right\rbrace \\
       &+ (\lambda)_{n-2} \frac{1}{(n-3)!} (\beta x)^{n-3} \left\lbrace (\beta | \alpha)_3 x + 6\beta^2(\beta - \alpha)x^2 + 6\beta^3 x^3 \right\rbrace \\
       &+ (\lambda)_{n-2} \frac{1}{2!(n-4)!} (\beta x)^{n-4} \left\lbrace \beta(\beta - \alpha)x + 2\beta^2 x^2 \right\rbrace
       \end{split}																		\end{equation*}
       \begin{equation*}
       \begin{split}
       T_{n}^{\lambda,x} (\alpha,\beta,0) &\sim (\lambda)_n \beta^n x^n + (\lambda)_{n-1} n(n-1) (\beta x)^{n-2} \left\lbrace \beta(\beta - \alpha)x + 2\beta^2 x^2 \right\rbrace \\
       &+ (\lambda)_{n-2} n(n-1)(n-2) (\beta x)^{n-3} \left\lbrace (\beta | \alpha)_3 x + 6\beta^2(\beta - \alpha)x^2 + 6\beta^3 x^3 \right\rbrace \\
       &+ (\lambda)_{n-2} \frac{n(n-1)(n-2)(n-3)}{2} (\beta x)^{n-4} \left\lbrace \beta(\beta - \alpha)x + 2\beta^2 x^2 \right\rbrace
       \end{split}
       \end{equation*}
       \begin{equation*}
       \begin{split}
       T_{n}^{\lambda,x} (\alpha,\beta,0) &\sim (\lambda)_n \beta^n x^n + (\lambda)_{n-1} (n)_2 (\beta x)^{n-2} \left\lbrace \beta(\beta - \alpha)x + 2\beta^2 x^2 \right\rbrace \\
       &+ (\lambda)_{n-2} (n)_3 (\beta x)^{n-3} \left\lbrace (\beta | \alpha)_3 x + 6\beta^2(\beta - \alpha)x^2 + 6\beta^3 x^3 \right\rbrace \\
       &+ (\lambda)_{n-2} \frac{(n)_4}{2} (\beta x)^{n-4} \left\lbrace \beta(\beta - \alpha)x + 2\beta^2 x^2 \right\rbrace
       \end{split}
       \end{equation*}
       where $n=o(\lambda^{1/2})$, $\lambda \rightarrow \infty$. Now, we look at the general case of $T_{n}^{\lambda,x} (\alpha,\beta,\gamma)$. Obtain few terms of the asymptotic expansion.
       \begin{equation*}
       \frac{T_{n}^{\lambda,x} (\alpha,\beta,\gamma)}{(\lambda)_n (n!)} = \frac{W(n,0)}{(\lambda - n)_0} + \frac{W(n,1)}{(\lambda - n + 1)_1} 
       + \frac{W(n,2)}{(\lambda - n + 2)_2} + o\left( \frac{W(n,2)}{(\lambda - n + 2)_2} \right)
       \end{equation*}
       \begin{align*}
       a_1 &= \sum_{k=0}^{1} \beta^k k! S(1,k;\alpha,\beta,\gamma)x^k \\
       		&= S(1,0;\alpha,\beta,\gamma) + \beta S(1,1;\alpha,\beta,\gamma)x \\
       		&= \gamma + \beta x 
       \end{align*}
       \begin{align*}
       a_2 &= \sum_{k=0}^{2} \beta^k k! S(2,k;\alpha,\beta,\gamma)x^k \\
       		&= S(2,0;\alpha,\beta,\gamma) + \beta S(2,1;\alpha,\beta,\gamma)x + \beta^2 2! S(2,2;\alpha,\beta,\gamma)x^2 \\
       		&= (\gamma | \alpha)_2 + \beta S(2,1;\alpha,\beta,\gamma)x + \beta^2 2! x^2 \\
       		&= (\gamma | \alpha)_2 + \beta(\beta - \alpha + 2\alpha)x + 2\beta^2 x^2\\
       a_3 &= \sum_{k=0}^{3} \beta^k k! S(3,k;\alpha,\beta,\gamma)x^k \\
       		&= S(3,0;\alpha,\beta,\gamma) + \beta S(3,1;\alpha,\beta,\gamma)x + 2!\beta^2 S(3,2;\alpha,\beta,\gamma)x^2\\
       		& + 3!\beta^3 S(3,0;\alpha,\beta,\gamma)x^3 \\
       		&= (\gamma|\alpha)_3 + \left[(\gamma|\alpha)_2 + (\beta - 2\alpha + \gamma)(\beta - \alpha + 2\gamma)\right]\beta x\\ 
       		&+ 6(\beta - \alpha + \gamma)\beta^2 x^2 + 6\beta^3 x^3
       \end{align*}
       \begin{align*}
       W(n,0) &= \frac{(\gamma + \beta x)^n}{n!} \\
       W(n,1) &= \frac{1}{(n-2)!} (\gamma + \beta x)^{n-2} \left\lbrace (\gamma|\alpha)_2 + \beta(\beta - \alpha + 2\gamma)x + 2\beta^2 x^2 \right\rbrace \\
       W(n,2) &= \frac{1}{(n-3)!} (\gamma + \beta x)^{n-3} \left\lbrace (\gamma|\alpha)_3 + \left[ (\gamma|\alpha)_2 + (\beta - 2\alpha + \gamma)(\beta - \alpha + 2\gamma) \right]\beta x \right\rbrace \\
       &+ \frac{1}{(n-3)!} (\gamma + \beta x)^{n-3} \left[ 6(\beta - \alpha + \gamma)\beta^2 x^2 + 6\beta^3 x^3 \right] \\
       &+ \frac{1}{2(n-4)!} (\gamma + \beta x)^{n-4} \left\lbrace (\gamma|\alpha)_2 + \beta(\beta - \alpha + 2\gamma)x + 2\beta^2 x^2  \right\rbrace^2
       \end{align*}
       Thus, 
       \begin{equation*}
       \begin{split}
       &T_{n}^{\lambda,x} (\alpha,\beta,\gamma) \sim n! \left[(\lambda)_n W(n,0) + (\lambda)_{n-1} W(n,1) + (\lambda)_{n-2} W(n,2) \right] \\
       													&= (\lambda)_n (\gamma + \beta x)^n + (\lambda)_{n-1} (n)_2 (\gamma + \beta x)^{n-2} \left\lbrace (\gamma|\alpha)_2 + \beta(\beta - \alpha + 2\gamma)x + 2\beta^2 x^2 \right\rbrace \\
       													&+ (\lambda)_{n-2} (n)_3 (\gamma + \beta x)^{n-3} \left\lbrace (\gamma|\alpha)_3 + \left[ (\gamma|\alpha)_2 + (\beta - 2\alpha + \gamma)(\beta - \alpha + 2\gamma) \right]\beta x \right\rbrace\\
       													&+ (\lambda)_{n-2} (n)_3 (\gamma + \beta x)^{n-3} \left\lbrace 6(\beta - \alpha + \gamma)\beta^2 x^2 + 6\beta^3 x^3\right\rbrace \\
       													&+ (\lambda)_{n-2} (n)_4 (\gamma + \beta x)^{n-4} \left\lbrace (\gamma|\alpha)_2 + \beta(\beta - \alpha + 2\gamma)x + 2\beta^2 x^2  \right\rbrace^2.
       \end{split}																				
       \end{equation*}

      \end{section}

\end{document}